\documentclass[12pt,oneside,reqno]{amsart}
\usepackage{amssymb,latexsym, amsmath, amscd, array, graphicx}
\usepackage{amssymb}
\usepackage{amsmath}
\theoremstyle{plain}

\newtheorem{thm}{Theorem}[section]

\newtheorem{lem}[thm]{Lemma}

\newtheorem{prop}[thm]{Proposition}

\newcommand\di[1]{\hbox{$#1$-dim}}

\theoremstyle{definition}
\newtheorem{defin}[thm]{Definition}

\newtheorem{rem}[thm]{Remark}
\newtheorem{remark}[thm]{Remark}

\newtheorem{ex}[thm]{Example}

\DeclareMathOperator{\as}{{\rm asdim}}
\DeclareMathOperator{\ANas}{{\rm AN-asdim}}
\DeclareMathOperator{\las}{{\rm \ell-asdim}}
\def\cu{{\mathcal U}}

\def\ovx{{\overline{x}}}
\def\ovy{{\overline{y}}}

\makeatletter
\renewcommand{\paragraph}{%
\@startsection{paragraph}{4}%
{\z@}{0.80ex \@plus 1ex \@minus .2ex}{-1em}%
{\normalfont\normalsize\bfseries}%
}
\makeatother

\begin{document}
\title{\textsc{On Asymptotic Dimension with Linear Control}}
\author{Corry M Bedwell}

\address{Corry Bedwell, Department of Mathematics, University
of Florida, 358 Little Hall, Gainesville, FL 32611-8105, USA}
\email{corrybedwell@gmail.com}


\begin{abstract} We construct a countable p-local group with a proper invariant metric
whose Assouad-Nagata dimension is strictly greater that the asymptotic dimension with
linear control. This solves Problem 8.6 from the list~\cite{Dr}.

We study asymptotic dimension with linear control $\las_{\omega}$ that depends on 
a fixed ultrafilter $\omega$
on $\mathbb N$. It turns out that the asymptotic Assouad-Nagata dimension is the supremum of
$\las_{\omega}$ on all $\omega$ and the asymptotic dimension with linear control is the minimum
of $\las_{\omega}$ over all $\omega$.
\end{abstract}

\maketitle

\section{Introduction}

The asymptotic dimension was defined by Gromov to study the finitely generated groups~\cite{Gr}.
His definition can be applied to general metric spaces, though we prefer to consider discrete metric spaces. 
One of the main examples of metric spaces are graphs with the length one of each edge. Our main examples are  the sets of vertices of  graphs possibly with a rescaled metric. In particular, we consider the discrete interval of the length $n$, a metric space $I(n)$ isomorphic to $[0,n]\cap\mathbb N$ and the discrete circle $S(n)$ of the length $n$, i.e. the set of vertices
of a cycle graph of length $n$. For $a>0$ by $I_a(n)$ and $S_a(n)$ we denote the spaces $I(n)$ and
$S(n)$ with the metrics multiplied by $a$.

{\em The asymptotic dimension} $\as X$ of a metric space $X$
does not exceed $n$, $\as X\le n$ if for any $\lambda<\infty$ there are $n+1$ uniformly bounded $\lambda$-disjoint families $\cu^0,\dots,\cu^n$ of subsets of $X$ such that $\cu^0\cup\dots\cup\cu^n$ covers $X$.
Thus, in the above definition, $\cu^i=\{U_{\alpha}^i\}_{\alpha\in A}$ and
$dist(U_{\alpha}^i,U_{\beta}^i) > \lambda$ for $\alpha\ne\beta$,
and there is $D<\infty$ such that ${\mathrm{diam}}(U_\alpha^i)\le D$ for  all
$\alpha\in A$ and all $i$~\cite{DS}. 

In this case we say that the dimension of $X$ on a scale $\lambda$ with the control $D$
does not exceed $n$, $\di{(\lambda,D)}X\le n$. We say $\di{(\lambda,D)}X=n$ if $\di{(\lambda,D)}X\le n$ and
the conditions for $\di{(\lambda,D)}X\le n-1$ cannot be fulfilled.

\begin{ex} \label{ex}
$$\di{(a,na)}I_a(k)=1=\di{(a,na)}S_a(l)$$ for $k\ge n+1$ and $l\ge 2n+1$, while
$$
 \di{(\lambda,0)}I_a(m)=0=\di{(\lambda,0)}S_a(r)$$
 for $\lambda<a$ and any $m,r\in\mathbb N$.
\end{ex}

The {\em asymptotic dimension with linear control}~\cite{Dr} $\las_* X$ of a metric space $X$ is defined as follows:
$\las_* X\le n$ if there is $c>0$ such that for every $R<\infty$ there is $\lambda>R$ such that
$\di{(\lambda,c\lambda)}X\le n$.

We note that in~\cite{Dr} this dimension was denoted as $\las X$.

The {\em asymptotic Assouad-Nagata dimension }(~\cite{DH}~\cite{DSm}~\cite{BDLM}) $\ANas X\le n$ if there is $c>0$ such that
$\di{(\lambda,c\lambda)}X\le n$ for $\lambda>r_0$ for some $r_0$. We note that for discrete metric spaces
where the distance between any pair of distinct points is greater than some fixed positive number the 
condition $\lambda>r_0$ can be dropped. In that case the asymptotic Assouad-Nagata dimension coincides with
the classical  Assouad-Nagata dimension~\cite{As} $\dim_{AN}X$.

{\bf Problem} (Problem 8.6 \cite{Dr}).
{\em Does ${\las}_* X=\ANas X$? What if $X$ is a finitely generated group?}

In this paper we give a negative answer to the first part of this question. Our counter-example is
a countable p-local group with a proper invariant metric. The case of finitely generated groups remains open.

The decisive property of our counter-example $X$ is that the function
$f(\lambda)=\di{(\lambda,c\lambda)}X$ has different limits with respect to different
ultrafilters $\omega$ on $\mathbb N$.
In the second part of the paper we investigate what kind of dimension of $X$ we obtain by taking the limits of
$f(\lambda)=\di{(\lambda,c\lambda)}X$ with respect to an ultrafilter.
We denote such invariant of as $\las_{\omega}(X)$ and
 show that $$\las_* X=\min_{\omega}\{\las_{\omega} X\}$$
and $$\ANas X=\sup_{\omega}\{\las_{\omega} X\}.$$

\textbf{Acknowledgments:} I would like to thank Alexander Dranishnikov for his valuable advice throughout this paper.

\section{Counter-example}

We begin by constructing a rather tame metric space. The motivation for starting here is to highlight and extract the core idea that gives rise to the counterexample via a group. 

For two subsets $U,V\subset X$ of a metric space $(X,d)$ we demote by $d(U,V)=\inf\{d(u,v)\mid u\in U, v\in V\}$.
\begin{defin}
Let $A$ be a collection of subsets of a metric space $X$. Then for $\lambda > 0 $ we say $A$ is \textit{$\lambda$-disjoint} if $d(U,V) > \lambda$ for all $U,V \in A$ and $U \not = V$. Also for  $D \geq 0$ we say $A$ is \textit{$D$-uniformly bounded} if $diam(U) \leq D$ for all $U \in A$.  
\end{defin}

We call a metric space {\em $\lambda$-discrete} if $d(x,y)\ge\lambda$ for all $x\ne y$.
\paragraph*{}
Our goal is to recursively construct a sequence of finite metric spaces 
$$
Y_1\subset Y_2\subset\dots\subset Y_n\subset\dots$$ and isometric embeddings
and a monotone sequence of tending to infinity real numbers $(a_n)_{n=1}^{\infty}$  such that for $X_n=Y_n\setminus Y_{n-1}$ the
following conditions are satisfied:
\begin{itemize}
\item[(I)] $X_n$ is $a_n$-discrete;
\item[(II)] $\di{(a_n,ca_n)}X_n \geq 1$ for  $c \le n$;
\item[(III)] $diam(Y_{n-1})< a_n$;
\item[(IV)] $d(Y_{n-1},X_n)\ge a_n$.
\end{itemize}

\begin{prop}\label{idea}
Let $X = \bigcup _{n=1}^{\infty}Y_n=\coprod_{n=1}^{\infty}X_n$ be given the natural metric that comes from $Y_n$s. Then $\las_{*} X = 0$ and $\ANas X\ge 1$.
\end{prop}
\begin{proof}
The conditions  $(I)$,$(III)$, and $(IV)$, imply that for the sequence $\lambda_n = a_n -1$, $$\di{(\lambda_n,\lambda_n)}X = 0.$$ Indeed, the cover $\mathcal U$ that consists of $Y_{n-1}$, an $(a_n-1) $-disjoint $(a_n-1)$-bounded 
cover $\mathcal V=\{V_i\}$ of $X_n$, and singletons in $X\setminus Y_n$ is $(a_n-1)$-disjoint and $(a_n-1)$-bounded.
Hence $\las_{*} X = 0$. 

On the other hand $(II)$  forces $\ANas X \geq 1$. Indeed, if $\ANas X <1$, then there is $c>0$ such that $\di{(\lambda,c\lambda)}X\le 0$. Therefore, $\di{(\lambda,c\lambda)}X_n\le 0$ for all $n$.
This contradicts with (II) for $n\ge c$.
\end{proof}

Now we construct the sequences of metric spaces and real numbers that satisfy these properties. 
We set $a_1 = 1$ and $X_1=\{1\}\times\{0,1,2\}\subset\mathbb R\times\mathbb R$. Assume that $a_{n-1}$ 
and $Y_{n-1}\subset[0,a_{n-1}]\times\mathbb R$ have been defined for some $n$. Choose $a_n$ so that 
$a_n > diam (Y_{n-1})+1$.
Let $$X_n = \{\sum_{i=1}^{n}a_{i}\} \times I_{a_n}(n+1)\subset\mathbb R\times\mathbb R.$$ 
Note that the conditions I and III-IV are obviously satisfied.  The condition (II) follows from Example~\ref{ex}.

We set $Y_n=Y_{n-1}\cup X_n$ and note that $Y_n\subset[0,a_n]\times\mathbb R$.

We give our space
$X = \bigcup_{i=1}^{\infty}X_i \subset\mathbb{R}^2$
the inherited metric from $\mathbb{R}^2$ endowed with the $\ell_1$ metric. Figure 1 below depicts the first few stages of construction of $X$. 
\paragraph*{} 
\setlength{\unitlength}{10cm}
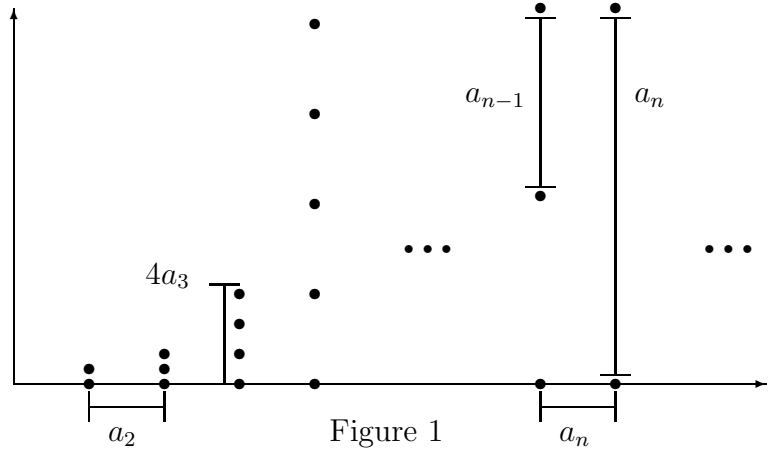
\begin{picture}(2,.75)
\put(.23,.1){\vector(1,0){1}}
\put(.23,.1){\vector(0,1){.5}}
\put(.33,.1){\circle*{.015}}
\put(.33,.12){\circle*{.015}}
\put(.33,.05){\line(0,1){.04}}
\put(.33,.07){\line(1,0){.1}}
\put(.43,.05){\line(0,1){.04}}
\put(.93,.05){\line(0,1){.04}}
\put(.93,.07){\line(1,0){.1}}
\put(1.03,.05){\line(0,1){.04}}
\put(.355,.025){$a_2$}
\put(.955,.025){$a_n$}
\put(.43,.1){\circle*{.015}}
\put(.43,.12){\circle*{.015}}
\put(.43,.14){\circle*{.015}}
\put(.405,.234){$4 a_3$}
\put(.51,.1){\line(0,1){.1325}}
\put(.49,.2325){\line(1,0){.04}}
\put(.53,.1){\circle*{.015}}
\put(.53,.14){\circle*{.015}}
\put(.53,.18){\circle*{.015}}
\put(.53,.22){\circle*{.015}}
\put(.63,.1){\circle*{.015}}
\put(.63,.22){\circle*{.015}}
\put(.63,.34){\circle*{.015}}
\put(.63,.46){\circle*{.015}}
\put(.63,.58){\circle*{.015}}
\put(.755,.28){\circle*{.0105}}
\put(.78,.28){\circle*{.0105}}
\put(.805,.28){\circle*{.0105}}
\put(.83,.475){$a_{n-1}$}
\put(.93,.1){\circle*{.015}}
\put(.93,.35){\circle*{.015}}
\put(.91,.3625){\line(1,0){.04}}
\put(.93,.3625){\line(0,1){.225}}
\put(.91,.5875){\line(1,0){.04}}
\put(.93,.6){\circle*{.015}}
\put(1.03,.1){\circle*{.015}}
\put(1.01,.1125){\line(1,0){.04}}
\put(1.03,.1125){\line(0,1){.475}}
\put(1.01,.5875){\line(1,0){.04}}
\put(1.03,.6){\circle*{.015}}
\put(1.055,.475){$a_n$}
\put(1.155,.28){\circle*{.0105}}
\put(1.18,.28){\circle*{.0105}}
\put(1.205,.28){\circle*{.0105}}
\put(.65,.025){Figure 1}
\end{picture}

\paragraph*{}

\begin{remark}\label{rem}
Clearly, there are many ways to construct the above $X_n$. One of the options is to build $X$
as a subset of the infinite metric wedge of rays $\vee_n (\mathbb R_+)_n$. It suffices to take $X_n$ as before  to be
isometric to the discrete interval $I_{a_n}(n+1)$. To satisfy the condition (IV) one should shift $I_{a_n}(n+1)$ along the ray $(\mathbb R_+)_n$ by $a_n$.
Then by Proposition~\ref{idea} $X=\vee_nX_n^+$ satisfies $\las_{*} X = 0$ and $\ANas X\ge 1$ where
$X_n^+=X_n\cup 0$ and $0\in(\mathbb R_+)_n$ is the wedge point.

Moreover, the the conditions I-IV are satisfied if every $X_n^+$ is replaced by an $a_n$-discrete circle $S_{a_n}(2n+2)$ with a base point. 
Then $X$ would be  the infinite wedge of discrete circles of increasing radii, $\vee_nS_{a_n}(2n+2)$.
\end{remark}

Let $\mathbb{Z}_{m}$ denote the group of integers modulo $m$. It is generated by one element $\bar 1$ with corresponding Cayley graph $X_m$ being a circle with $m$ edges.  Then the group $\mathbb Z_m$ with the word metric can be identified with the discrete circle $S(m)$. We will denote the distance between $x,y\in\mathbb Z_m$ in this metric as $|x-y|_m$. Then the {\em $a$-weighted metric} $d_a$ on $\mathbb Z_m$ is given by the formula $d_a(x,y)=a|x-y|_m$. Note that $0$ is a natural base point for $\mathbb Z_m$.
Then the last part of Remark~\ref{rem} can be deformed into the following:

\begin{prop}\label{wedge} 
Given a prime number $p$, there is a
monotone sequence of natural numbers $(a_n)$ tending to infinity such that $\ANas(X)\ge 1$ and $\las_*(X)=0$ where $X=\vee_n(\mathbb Z_{p^{n}},d_{a_n})$
\end{prop}
\begin{proof}
We define $a_n$ recursively by the condition
$diam(\vee_{i=1}^{n-1}(\mathbb Z_{p^{i}},d_{a_i}))< a_n$. Then all the conditions I-IV will be satisfied.
To be formal, one needs the inequality $p^{n}\ge 2(n+1)$ which holds for $p>2$ and holds eventually for $p=2$. 
\end{proof}

Let $(X_n,d_n)$ be a sequence of metric spaces with base points $x_n^0$.  We define the metric space
$\bigoplus X_n$ to be the subset of $\prod X_n$ that consists of the $\omega$-tuples which are eventually the base points. We consider the $\ell_1$-metric on $\bigoplus X_n$: $$d((x_n),(y_n))=\sum_{n=1}^{\infty} d_n(x_n,y_n).$$

\begin{lem}\label{product} Suppose that $X_n$ is a sequence of metric spaces with base points and
$k\in\mathbb N$ is such that the metric spaces $X_j$ are $\lambda$-discrete for all $j > k$.  Then for any $D_0$ and $D \geq diam(\prod_{i=1}^{k-1}X_i)$
$$\di{(\lambda,D + D_{0})}(\bigoplus X_n)\leq \di{(\lambda,D_{0})}X_k $$
 for the $\ell_1$-metric on the product $\prod_{i=1}^{k-1}X_i$.
\end{lem}
\begin{proof}
Assume that $\di{(\lambda,D_{0})}X_k \leq n$. Let $\cu^{0}_k,...,\cu^{n}_k$ be the $\lambda$-dijiont  $D_0$-uniformly bounded families whose union cover $X_k$. For $i=0,\dots,n$ we define  a family of sets in $\bigoplus X_n$ as follows
$$
\cu^i=\{\prod_{i=1}^{k-1}X_i\times U\times (\bar x_{>k})\ \mid U\in\cu^i_k,\  (\bar x_{>k})\in\bigoplus_{n=k+1} ^{\infty}X_n\},
$$ 
where $(\bar x_{>k})= (x_{k+1},x_{k+2},...)$. Clearly the union of the families $\cu^0,\dots, \cu^n$ is a cover of $\bigoplus X_n$.

First we show that these families are uniformly $(D+D_0)$-bounded. Let $\bar y,\bar z\in\prod_{i=1}^{k-1}X_i\times U\times (\bar x_{>k})$ with $U\in \cu_k^i$.
Then,
$$d(\bar y,\bar z) = \sum^{k}_{i=1}d(z_i,y_i) = (\sum^{k-1}_{i=1}d(z_i,y_i)) + d_k(z_k,y_k) \leq diam(\prod_{i=1}^{k-1}X_i) + d_k(z_k,y_k) \leq D + D_0.$$

Next we show that each family $\cu^i$ is $\lambda$-disjoint. Let $\bar y\in\prod_{i=1}^{k-1}X_i\times U\times (\bar y_{>k})$ and $\bar z\in\prod_{i=1}^{k-1}X_i\times V\times (\bar z_{>k})$ with $U, V\in \cu_k^i$. 
We have to check two cases:
\begin{enumerate}
\item[\textbf{Case 1}]$U \not = V$. In this case we have that $$d_k(\bar y,\bar z) \geq d(y_k,z_k) > \lambda$$
as $\cu^{i}_k$ is $\lambda$-disjoint. 
\item[\textbf{Case 2}]$(\bar y_{>k})\not = (\bar z_{>k})$. In this case, there is $j > k$ so that $y_j \not = z_j$ and hence,
$$d(\bar y,\bar z) \geq d_j(w_j,z_j) > \lambda$$
as $X_j$ is $\lambda$-disjoint for all $j > k$.
\end{enumerate}
\end{proof}

Let $G = \bigoplus_{i=1}^{\infty} \mathbb{Z}_{p^{i}}$. For any sequence $(a_n)_{n=1}^{\infty}\subset\mathbb{N}$ we define the metric $d : G\times G \rightarrow \mathbb{R}$ by,

\begin{equation*}
d(\ovx,\ovy) = \sum_{i=1}^{\infty} a_{i}|x_{i}-y_{i}|_{p^i} 
\end{equation*}
We recall that a metric space is called {\em proper} if every closed ball in that metric is compact.
\paragraph*{}
\begin{prop} Let $G = \bigoplus_{i=1}^{\infty} \mathbb{Z}_{p^{i}}$. Then  the metric on $G$ defined above for $a_n\to\infty$ is proper and invariant.
\end{prop}
\begin{proof}
It is well known that a metric defined on a countably generated group by weighting its generators with a sequence of numbers tending to infinity is proper (see~\cite{Sm}). The invariance can be easily verified:
Let $\ovx,\ovy \in G$, then
$$
d((\ovx,\overline{g}),(\ovy,\overline{g})) = d(\ovx + \overline{g},\ovy + \overline{g}) = \sum_{i=1}^{\infty} a_{i}|(x_{i}+g_{i})-(y_{i}+g_{i})|_{p^i} =$$
$$ \sum_{i=1}^{\infty} a_{i}|x_{i}+g_{i}-g_{i}-y_{i})|_{p^i} = \sum_{i=1}^{\infty} a_{i}|x_{i}-y_{i}|_{p^i} = d(\ovx,\ovy).
$$
\end{proof}

We recall that a group $G$ is called {\em $p$-local} if every finitely generated subgroup $H\subset G$ is
a $p$-torsion group. 

\paragraph*{}
\begin{thm}\label{main}
There exists a countable $p$-local group $G$ with a proper invariant metric, such that $\ANas G$  differs from 
$\las_* G$.
\end{thm}
\begin{proof} Let $G$ be defined as above with the sequence $(a_n)$  as in Proposition~\ref{wedge}.
Note that $G=\bigoplus S_{a_n}(p^{n})$.

First we show that $\las_* G=0$. Let $R>0$ be given. Let $n$ be large enough so that $\lambda_n =a_n-1> R$  we show that $\di{(\lambda_n, 2\lambda_n)} G = 0$. Note that for each $i \geq n$ if $x,y \in \mathbb{Z}_{p^{i}}$ with $x\not = y$, then 
$$d(x,y) = a_i|x - y| \geq a_i \geq a_n > \lambda_n.$$
Also note that by construction of $G$, $$\lambda_n=a_n-1 \geq \prod_{i=1}^{n-1}\mathbb Z_{p^i}.$$ Therefore by Lemma~\ref{product},
\begin{equation}
\di{(\lambda_n,2\lambda_n)}G \leq \di{(\lambda_n, \lambda_n)} (S_{a_n}(p^n))=0.
\end{equation} 
For the last equality see Example~\ref{ex}.

Note that $\ANas G \ge\ANas\vee S_{a_n}(p^n)\geq 1$ by Proposition~\ref{wedge}. \end{proof}  

\section{Dimension that depends on ultrafilter}

\paragraph*{}
As a convention and for notational ease, we let $\mathbb{N}^* = \beta\mathbb{N}\setminus \mathbb{N}$, the set of non-principal ultrafilters on $\mathbb{N}$. Throughout the paper $\omega$ will refer to a non-principal ultrafilter and any unspecified ultrafilters are taken to be non-principal.

\paragraph*{}
For a fixed $c > 0$ and a metric space $(X,d)$ the function
$$f_c : \mathbb{N} \rightarrow \mathbb{N} \cup \{\infty\}=\alpha\mathbb N$$ defined
$f_c(\lambda) =\di{(\lambda,c\lambda)}X$  extends to the Stone-\v Cech compactification, 
$$\overline{f}_c : \mathbb{N}^* \rightarrow \alpha\mathbb N$$
where $\alpha\mathbb N$ is the one-point compactification and $\overline{f}_c(\omega) = \lim_{\omega}f_c$.

Finally we define asymptotic dimension with linear control of a metric space $X$ with respect to an ultrafilter $\omega\in\mathbb{N}^*$.
$$\las_{\omega}X=\min_c\{\bar f_c(\omega)\}.$$

In particular the definition implies that if $\las_{\omega}X = n$, then there is a $c > 0$ such that 
$$\lim_{\omega}f_{c} = n$$

\paragraph*{}
We note that  $\las_{*}X$ and $\ANas X$ are quasi-isometric invariant so that it suffices to check the $\lambda$-scaled dimension of $X$ over the $\lambda$ in $\mathbb{N}$ instead of over $\mathbb{R}$. 

\paragraph*{}
\begin{thm} For every metric space $X$ the following equality holds,
 $$\ANas X=sup\{\las_{\omega}(X) : \omega \in \mathbb{N}^*\}.$$
\end{thm}
\begin{proof} 
Assume that $\ANas X \leq k$. Then there is $c>0$ and $r\in\mathbb{N}$ such that for each $\lambda\geq r$,
$$\di{(\lambda,c\lambda)}X = k.$$
Let $A = [r,\infty)\cap \mathbb{N}$ and let $\omega \in \mathbb{N}^*$ be arbitrary. Note that  $A\in\omega$. Taking $c$ from above it follows that
$$\overline{f}_c(\omega) = \lim_{\omega}f_c \leq k.$$ Therefore $\las_{\omega}X \leq k$. Since $\omega$ was arbitrary, we obtain $$sup\{\las_{\omega}(X) : \omega \in \mathbb{N}^*\}\leq k.$$
Now assume that $sup\{\las_{\omega}(X) : \omega \in \mathbb{N}^*\} = k$. Fix $n\in\mathbb{N}$ and define $A_n = \{(\overline{f}_n)^{-1} (\{0,...,k\})\}$. Set, 
$$A = \bigcup_{n\in\mathbb{N}}A_n$$
It is clear that $A$ forms an open cover of $\mathbb{N}^*$. Since $\mathbb{N}^*$ is compact, $A$ admits a finite sub-cover, say
$$\{(\overline{f}_{n_{1}})^{-1} (\{0,...,k\}),...,(\overline{f}_{n_{j}})^{-1} (\{0,...,k\})\}$$
Take $c=max\{n_{1},...,n_{j}\}$ and set $A^{'} = (\overline{f}_{n_{1}})^{-1} (\{0,...,k\})\cup,...,\cup(\overline{f}_{n_{j}})^{-1} (\{0,...,k\})$. Let $B = \beta\mathbb{N}\setminus A^{'}$. Observe that $B$ is closed in $\beta\mathbb{N}$ and hence compact, and further $B$ is a subset of $\mathbb{N}$. Thus $B$ is finite. As a consequence we obtain that there is an $r > 0$ so that $[r,\infty)\cap\mathbb{N}\subset A^{'}$. Therefore taking $c>0$ from above we see,
$$f_c(\lambda) = \di{(\lambda,c\lambda)}X \leq k$$
For all $\lambda \in [r,\infty)\cap\mathbb{N}$.That is, $\ANas X \leq k$. 
\end{proof}

\paragraph*{}
\begin{thm} For every metric space $X$ the following equality holds, $$\las_* X=min\{\las_{\omega}(X) : \omega \in \mathbb{N}^*\}.$$
\end{thm}
\begin{proof} 
Assume that $\las_* X=k$. Then there is $c>0$ such that for each $i\in\mathbb{N}$ there is a $\lambda_i \geq i$ with
$$f_c(\lambda_i)= \di{(\lambda,c\lambda)}X \leq k$$ 
Let $A = \{\lambda_i\}_{i=1}^{\infty}$. There is $\omega\in\mathbb{N}^*$ so that $A \in\omega$. Then from this we have that 
$$\lim_{\omega}f_c \leq k$$ 
Hence $\las_{\omega}(X) \leq k$. Taking minimums on both sides we obtain the desired result,
$$min\{\las_{\omega}(X) : \omega \in \mathbb{N}^*\} \leq k.$$
On the other hand assume that $min\{\las_{\omega}(X) : \omega \in \mathbb{N}^*\} = k$. Then there is $\omega \in \mathbb{N}^*$ so that 
$$\las_{\omega}(X) = k.$$
Which by definition yields that there is a $c>0$ so that 
$$\lim_{\omega}f_c = k$$
Let $A = (f_c)^{-1}(\{k\})$. Note that as $\omega$ is non-principal, necessarily $A$ is infinite. Therefore it follows that for all $i\in\mathbb{N}$ there is a $\lambda_i \in A$ with $\lambda_i \geq i$ and
$$f_c(\lambda_i)= \di{(\lambda,c\lambda)}X \leq k$$
That is $\las_* X \leq k$, which completes the proof.
\end{proof}

The natural question arises, for a given metric space $X$ which values of $\las_{\omega}$ of $X$ can be obtained. The first part of the paper established that there is a countable group under which this dimension obtains two values. We end the paper with one final remark.

\begin{rem}
Using the same methods as in Theorem~\ref{main} one can construct a countably generated group such that there exists an ultrafilter $\omega$ with,
$$\las_* G < \las_{\omega} G < \ANas G.$$
Moreover, one can any finite set of integers as value of the dimension $\las_{\omega} G$.
\end{rem}

\end{document}